\documentclass[12pt]{amsart}
\usepackage[colorlinks=true, pdfstartview=FitV, linkcolor=blue, citecolor=blue, urlcolor=blue]{hyperref}
\usepackage{amssymb}
\calclayout

\def\C{\mathbb {C}}
\def\R{\mathbb {R}}
\def\NN{\mathcal {N}}

\def\H{\mathbb {H}}

\def\inv{^{-1}}

\newcommand{\rank}{\operatorname{rank}}
\newcommand{\SL}{\operatorname{SL}}
\newcommand{\SU}{\operatorname{SU}}

\newcommand{\SO}{\operatorname{SO}}

\def\GL{\operatorname{GL}}

\def\phi{\varphi}

\def\I{\mathcal I}
\def\quot#1#2{#1/\!\!/#2}
\def\Dbar{\leavevmode\lower.6ex\hbox to 0pt{\hskip-.23ex
    \accent"16\hss}D}
\def\apr{{\operatorname{apr}}}
\def\pr{{\operatorname{pr}}}
\def\codim{{\operatorname{codim}}}

\numberwithin{equation}{subsection}

\newtheorem{theorem}[subsection]{Theorem}
\newtheorem{lemma}[subsection]{Lemma}
\newtheorem{proposition}[subsection]{Proposition}
\newtheorem{corollary}[subsection]{Corollary}

\theoremstyle{definition}

\theoremstyle{remark}

\newtheorem{remark}[subsection]{Remark}

\newtheorem{example}[subsection]{Example}

\title[Reduced invariant sets]{\boldmath Reduced invariant sets} 
 \author{Gerald W. Schwarz}
\address{Department of Mathematics\\
Brandeis University\\
Waltham, MA 02454-9110}
\email{schwarz@brandeis.edu}
\dedicatory{In honor of Dick Palais}

\begin{document}
\begin{abstract}
Let $K$ be a compact Lie group and $W$ a finite-dimensional real $K$-module. Let $X$ be a $K$-stable real algebraic subset of $W$. Let   $\I(X)$ denote the ideal of $X$ in $\R[W]$ and let  $\I_K(X)$  be the ideal generated by $\I(X)^K$. We find necessary conditions and sufficient conditions for $\I(X)=\I_K(X)$ and for $\sqrt{\I_K(X)}=\I(X)$. We consider  analogous questions for actions of  complex reductive groups.
 \end{abstract}

\maketitle
 
\section{Introduction}\label{sec:intro}

Let $K$ be a compact Lie group, let  $W$ be a finite-dimensional real $K$-module and let $X\subset W$ be  $K$-invariant and real algebraic (the zero set of real polynomial functions on $W$). Let $\I(X)$ denote the ideal of $X$ in $\R[W]$. Let $\R[W]^K$ denote the $K$-invariants in $\R[W]$ and let  $\I_K(X)$  be the ideal generated by $\I(X)^K:=\I(X)\cap\R[W]^K$.  We say that $X$ is \emph{$K$-reduced\/} if $\I_K(X)=\I(X)$ and \emph{almost $K$-reduced\/} if $\sqrt{\I_K(X)} = \I(X)$.  Let $Kw$ be an orbit in $W$. Then the \emph{slice representation at $w$\/} is the action of the isotropy group $K_w$ on $N_w$, where $N_w$ is a $K_w$-complement to $T_w(Kw)$ in $W\simeq T_w(W)$.
An orbit $Kw\subset W$ is \emph{principal} (resp.\ \emph{almost principal}) if the  image of $K_w$ in $\GL(N_w)$ is trivial (resp.\ finite). We denote the principal (resp.\ almost principal) points of $W$ by $W_\pr$ (resp.\ $W_\apr$) and we set $X_\pr:=W_\pr\cap X$ and $X_\apr:=W_\apr\cap X$. The \emph{strata\/} of $W$ are the collections of points $S\subset W$ whose isotropy groups are conjugate. There are finitely many strata.  If $\R[W]$ is a free $R[W]^K$-module, then we say that $W$ is \emph{cofree}.  In the following, when we talk about one set being dense in another, we are referring to the Zariski topology.

Here are  our main results:

\begin{theorem}\label{thm:mainreal1}
 If $X$ is  $K$-reduced (resp.\ almost $K$-reduced), then $X_\pr$ (resp.\ $X_\apr$) is  dense in $X$, and conversely if $W$ is cofree.
\end{theorem}

\begin{theorem}\label{thm:mainreal2}
Let $w\in W$. Then the orbit $Kw$ is $K$-reduced (resp.\ almost $K$-reduced) if and only if $Kw$ is principal (resp.\ almost principal).
\end{theorem}

To prove the results above and to obtain further results we need to complexify. Let $V=W\otimes_\R\C$ and $G=K_\C$ be the complexifications of $W$ and $K$. We have the quotient morphism
$\pi\colon V\to\quot VG$ where $\pi$ is surjective, $\quot VG$ is an affine variety and $\pi^*\C[\quot  VG]=\C[V]^G$. We have the Luna strata of the quotient $\quot VG$  whose inverse images in  $V$ are the strata of $V$. The strata of $V$ are in 1-1 correspondence with those of $W$ \cite[\S 5]{SchLifting}. 
Let $Y=X_\C$ be the complexification of $X$ (the Zariski closure of $X$ in $V$). We say that $Y$ is \emph{$G$-saturated\/} if $Y=\pi\inv(\pi(Y))$ and that \emph{$Y$ is $G$-reduced\/} if the ideal $\I(Y)$ of $Y$ is generated by $\I(Y)^G$. We can define $Y_\apr$ and $Y_\pr$ as above (see \S \ref{sec:stable}). If $f_1,\dots,f_k$ are functions on a complex variety, let $\I(f_1,\dots,f_k)$ denote the ideal they generate.

\begin{theorem}\label{thm:complexif}
\begin{enumerate}
\item $X$ is almost $K$-reduced if and only if $Y$ is $G$-saturated.
\item $X$ is $K$-reduced if and only if $Y$ is $G$-reduced.
\item $X_\apr$ (resp.\ $X_\pr$) is dense in $X$ if and only if $Y_\apr$ (resp.\ $Y_\pr)$ is dense in $Y$.
\end{enumerate}
\end{theorem}

\begin{theorem} \label{thm:complex} Assume that $\quot YG\subset \quot VG$ is the zero set of $f_1,\dots,f_k$.  
\begin{enumerate}
\item  Suppose that $Y_\apr$ is dense in $Y$ and that for any stratum $S$ of $V$ which intersects  $Y\setminus Y_\apr$ the codimension of $S$ in $V$ is at least $k+1$. Then $Y$ is $G$-saturated..
\item Suppose that $Y_\pr$ is dense in $Y$ and that $Y$ is $G$-saturated. In addition, suppose that $\I(\pi(Y))=\I(f_1,\dots,f_k)$ where $Y$ has codimension $k$ in $V$. Then $Y$ is $G$-reduced.
\end{enumerate}
\end{theorem}
  
\begin{corollary}\label{cor:main}
 If (1) above holds, then $X$ is almost $K$-reduced.
 If (2) holds, then $X$ is $K$-reduced.
\end{corollary}

In sections \ref{sec:complex}--\ref{sec:unstable} we consider when a general $G$-invariant $Y\subset V$ is $G$-saturated or $G$-reduced and we establish Theorem \ref{thm:complex}. In section \ref{sec:real} we treat the real case by complexifying. At the end of section \ref{sec:real}     we establish Theorems  \ref{thm:mainreal1}, \ref{thm:mainreal2}
 and \ref{thm:complexif}.
  
  D.\ \v Z.\ \Dbar okovi\'c posed the question of identifying the $X$ which are $K$-reduced. Our results give a partial answer. We thank  M. Ra{\"\i}s for transmitting the question to us. We thank the referee for a careful reading of the manuscript, helpful suggestions and Lemma \ref{lem:referee}.

\section{The complex case} \label{sec:complex}  Let $G$ be a complex reductive group and $Y$ an affine algebraic set with an algebraic $G$-action. Dual to the inclusion $\C[Y]^G\subset \C[Y]$  we have the quotient morphism $\pi_Y\colon Y\to\quot YG$. 
Let $V$ be a finite-dimensional $G$-module and let $Y$ be a $G$-stable algebraic subset of $V$ (the zero set of an ideal of $\C[V]$).
We shall denote $\pi_V$ simply by $\pi$. Then $\pi_Y=\pi|_Y$ and  $\pi(Y)\simeq\quot YG$ is an algebraic subset of $\quot VG$. We say that $Y$ is \emph{$G$-saturated\/} if $Y=\pi\inv(\pi(Y))$. Let $\I(Y)$ denote the ideal of $Y$ in $\C[V]$ and let $\I_G(Y)$ denote the ideal generated by $\I(Y)^G$. We say that $Y$ is \emph{$G$-reduced} if $\I(Y)=\I_G(Y)$. The \emph{null cone $\NN(V)$ of $V$\/}  is the   fiber $\pi\inv(\pi(0))$. Then $\NN(V)$ is (scheme theoretically) defined by the ideal $\I_G(\{0\})$  so that the scheme $\NN(V)$ is reduced if and only if the set $\NN(V)$ is $G$-reduced, in which case we say that   $V$ is \emph{coreduced}. See \cite{KSreduced} for more on coreduced representations.

	The points of $\quot VG$ are in one-to-one correspondence with the closed $G$-orbits in $V$. The \emph{Luna strata\/} of $\quot VG$ are the sets of closed orbits whose isotropy groups are  all $G$-conjugate. There are finitely many strata in $\quot VG$, and we consider their inverse images in $V$ to be the  \emph{strata\/} of $V$.  Let $v\in V$ such that $Gv$ is closed. Then the isotropy group $G_v$ is reductive, and there is a $G_v$-stable complement $N_v$ to $T_v(Gv)$ in $V\simeq T_v(V)$. We call the action of $G_v$ on $N_v$ the \emph{slice representation at $v$\/}.   
	
	We start with some examples.	

\begin{example}\label{ex:sln}
Let $(V,G)=(k\C^n,\SL_n)$, $k\geq n$. The invariants are generated by the determinants $\det_{i_1,\dots,i_n}$ where  the indices $1\leq i_1<\dots< i_n\leq k$ tell us which $n$ copies of $\C^n$ to take. Then $V$ is coreduced since $\NN(V)$ is the determinantal variety of ($k\times n$)-matrices of rank at most $n-1$. See also \cite{KSreduced}. All orbits outside the null cone are closed with trivial isotropy group, hence are principal.
\end{example}

\begin{example}
Let $G\subset\GL(V)$ be finite and nontrivial.  Then $\NN(V)$ is the origin which is $G$-saturated but   not $G$-reduced.
\end{example}

Part (2) of the proposition below follows from Serre's criterion for reducedness  \cite[Ch.\ 7]{Matsumura}. Part (1) also follows, using the Jacobian criterion for smoothness.
\begin{proposition}\label{prop:Serre} Let $Y\subset V$ be a $G$-saturated  algebraic set.
\begin{enumerate}
\item If $Y$ is $G$-reduced, then for every irreducible component $Y_k$ of $Y$ there is a  point  of $Y_k$ where $\rank f =\codim\ Y_k$. Here $f=(f_1,\dots, f_d)\colon V\to\C^d$ and the $f_i$ generate $\I_G(Y)$.
\item If $\I_G(Y)=\I(f_1,\dots,f_d)$ where the $f_i\in\C[V]^G$ and $Y$ has codimension $d$, then $Y$ is $G$-reduced if and only if the rank condition of (1) is satisfied.
\end{enumerate}
\end{proposition}

\begin{example}\label{ex:so3}
Let $G=\SO_3(\C)$ acting as usual on $V=2\C^3$.  Then the invariants are generated by inner products $f_{ij}$, $1\leq i\leq j\leq 2$. Each copy of $\C^3$ has a weight basis $\{v_2,v_0,v_{-2}\}$ relative to the action of the maximal torus $T=\C^*$ where $v_j$ has weight $j$. The null cone $Y:=\NN(V)$ is the $G$-orbit of all the vectors $v=(\alpha v_2,\beta v_2)$ for $\alpha$, $\beta\in\C$. But one easily calculates that the rank of $(f_{11},f_{22},f_{12})\colon V\to\C^3$ at $v$ is at most 2 while $Y$ has codimension 3.  Thus the null cone   is not $G$-reduced. 
\end{example}

\section{The case where $Y_\pr$ or $Y_\apr$ is dense in $Y$} \label{sec:stable}
Throughout this section we assume that $V$ is a \emph{stable} representation of $G$, i.e., there is a nonempty open subset of closed orbits. This is always the case when $(V,G)=(W_\C,K_\C)$ is a complexification (\cite{LunaClosed} or \cite[Cor.\ 5.9]{SchLifting}). Let $Gv$ be a closed orbit.   We say that \emph{$Gv$ is principal\/} if the slice representation $(N_v,G_v)$ is a trivial representation and that  \emph{$Gv$ is almost principal\/} if $G_v\to\GL(N_v)$ has finite image. We denote the principal (resp.\ almost principal) points of $V$ by $V_\pr$ (resp.\ $V_\apr$). If $Y\subset V$ is $G$-stable, we set 
  $Y_\pr=Y\cap V_\pr$ and $Y_\apr=Y\cap V_\apr$.  Both $Y_\apr$ and $Y_\pr$ are open in $Y$. In general, the fiber of $\pi$ through a closed orbit $Gv\subset V$ is $G\times^{G_v}\NN(N_v)$ (the $G$-fiber bundle with fiber $\NN(N_v)$ associated to the $G_v$-principal bundle $G\to G/G_v$). Thus the fiber is set-theoretically the orbit if and only $\NN(N_v)$ is a point. This  happens if and only if the image $G_v\to\GL(N_v)$ is finite, i.e, $v\in V_\apr$. Hence $Y_\apr$ is always $G$-saturated. Similarly, the fiber is scheme-theoretically the orbit   if and only if $\NN(N_v)$ is schematically a point which is equivalent to $G_v$ acting trivially on $N_v$, i.e., we have $v\in V_\pr$. Hence $Y_\pr$ is always $G$-reduced. To sum up we have

\begin{proposition}\label{prop:Greduced} Let $Gv$ be a closed orbit and let $Y\subset V$ be a $G$-stable algebraic set.
\begin{enumerate}
\item If $Y=Y_\apr$, then $Y$ is $G$-saturated. In particular,  $Gv$ is $G$-saturated if and only if it is almost principal.
\item If $Y=Y_\pr$, then $Y$ is $G$-reduced. In particular,  $Gv$ is $G$-reduced if and only if it is  principal.
\item The fiber $\pi\inv(\pi(v))$ is $G$-reduced if and only if the slice representation $(N_v,G_v)$ is coreduced.
\end{enumerate}
 \end{proposition}

\begin{corollary}\label{cor:finiteisotropygroups}
If the isotropy groups of $G$ acting on $Y$ are all finite, then $Y$ is $G$-saturated and if $G$ acts freely on $Y$, then $Y$ is $G$-reduced.
\end{corollary}

	Of course, it is possible that  $Y$ is $G$-saturated (resp.\ $G$-reduced) even if $Y_\apr$ (resp.\ $Y_\pr$) is empty.  But in the case of a complexification $Y=X_\C$ it is  necessary for $G$-saturation (resp.\ $G$-reducedness) that $Y_\apr$ (resp.\ $Y_\pr$) is dense in $Y$ (see \S \ref{sec:real}). We consider the case that $Y_\apr$ or $Y_\pr$ is not dense in $Y$ in the next section.
 
 Unfortunately, we do not have the analogues of Proposition \ref{prop:Greduced}(1) and (2) for $X$. See Example \ref{ex:bad} below.

 	 Recall that $V$ is \emph{cofree\/} if $\C[V]$ is a free $\C[V]^G$-module. Equivalently, $\pi\colon V\to\quot VG$ is flat, or $\C[V]^G$ is a regular ring and the codimension of $\NN(V)$ is $\dim\C[V]^G$ \cite[Proposition 17.29]{SchLifting}. 
	 
	We owe the following lemma to the referee.
	
\begin{lemma}\label{lem:referee} Let $V$ be a cofree $G$-module and let $U\subset\quot VG$ be locally closed.  
\begin{enumerate}
\item We have $\overline{\pi\inv(U)}=\pi\inv(\overline {U})$.
\item If $\pi\inv(U)$ is reduced, then so is $\pi\inv(\overline{U})$.
\end{enumerate}
\end{lemma}

\begin{proof}
For (1) set $Z:=\overline{\pi\inv(U)}$. Then $\pi(Z)$ is closed \cite[II.3.2]{KraftBook}, hence $\pi(Z)=\overline U$. Since $\pi$ is flat, so is $\pi\inv(\overline{U})\to\overline{U}$. Set $S:= \pi\inv(\overline{U})\setminus Z$. Then $S$ is open, hence $\pi(S)$ is open in $\overline{U}$ (by flatness). By construction, $\pi(S)$ does not meet $U$, hence we must have $S=\emptyset$, establishing (1).

For (2) we can assume that $U=\overline{U}_f=\{u\in \overline{U}\mid f(u)\neq 0\}$ for $f\in\C[\overline{U}]$. Set $Z:=\pi\inv(\overline{U})$, the schematic inverse  image of $\overline{U}$. Since $\C[\overline{U}]\to\C[U]=\C[\overline{U}]_f$ is injective and $\C[Z]$ is flat over $\C[\overline{U}]$, it follows that $\C[Z]\to\C[Z]_f=\C[\pi\inv(U)]$ is also injective. Since the latter ring is reduced, so is $\C[Z]$ and we have established (2).
\end{proof}

\begin{corollary}\label{cor:cofreealmost}
Suppose that $(V,G)$ is cofree and that $Y\subset V$ is a $G$-stable algebraic set such that $Y_{\apr}$ is dense in $Y$. Then $Y$ is $G$-saturated.
\end{corollary}

\begin{example}\label{ex:noncofree}
Let $(V,G)=(4\C^2,\SL_2)$ and let $Y=2\C^2\times\{0\}$. Then $Y_\pr=Y_\apr$ is dense in $Y$ (it is the set of linearly independent vectors in $Y$) but $Y$ is not $G$-saturated since it does not contain the null cone.  The $G$-module $V$ is not cofree, so we don't contradict Corollary \ref{cor:cofreealmost}. Note that this example is the complexification    of the case where $X=\C^2\times\{0\}\subset W:=\C^2\oplus\C^2$ and $K=\SU(2,\C)$. Thus $X_\pr=X_\apr$ is dense in $X$ but $X$ is not almost $K$-reduced. (We use Theorem \ref{thm:complexif}.)\ This   shows that cofreeness is also necessary in  Theorem \ref{thm:mainreal1}.
\end{example}

\begin{theorem}\label{thm:codimd}
Suppose that $Y\subset V$ is $G$-stable such that 
\begin{enumerate}
\item $Y_\apr$ is dense in $Y$.
\item $\quot YG\subset \quot VG$ is the zero set of $f_1,\dots,f_k$ where the minimal codimension  of a non almost principal stratum of $V$ which intersects $Y$ is at least $k+1$.
\end{enumerate}
 Then $Y$ is $G$-saturated.
\end{theorem}

\begin{proof}
Let $\widetilde Y$ denote $\pi\inv(\pi(Y))$. Then  each irreducible component of $\widetilde Y$ has codimension less than or equal to $k$. Let $S$ be a non almost principal stratum of $V$ which intersects $Y$. Then $S\cap \widetilde Y$ is nowhere dense in $\widetilde Y$. Thus $\widetilde Y_\apr$ is dense in $\widetilde Y$.   Now $\widetilde Y_\apr$ and $Y_\apr$ have the same image in $\quot YG$. Hence  $Y_\apr=\widetilde Y_\apr$ and $Y=\widetilde Y$ is saturated.
\end{proof}

\begin{example}
Let $(V,G)=(k\C^2,\SL_2)$, $k\geq 2$. The codimension of the null cone is $k-1$ and the subset $Y$ where the first copy of $\C^2$ is zero is not saturated, but corresponds to the subset of $\quot VG$ where the determinant invariants $\det_{12},\dots,\det_{1k}$ vanish (see Example \ref{ex:sln}). Thus the codimension condition  in Theorem \ref{thm:codimd}(2) is sharp.
\end{example}

Here is an example that is a complexification. 

 \begin{example}\label{ex:so2}
Let $(V,G)=(2\C^2,\SO_2(\C))$ and let $Y=\C^2\times\{0\}\cup \{0\}\times \C^2$. Then $Y_\apr$ is dense in $Y$ since any  point not in $\NN(V)$ is on a principal orbit and $\NN(V)$ is nowhere dense in $Y$. However, $Y$ is not  $G$-saturated since it does not contain $\NN(V)$.  Note that $\I(\quot YG)$ is generated by $\det$ (the determinant),     $f_{12}$ and $f_{11}f_{22}$ where the $f_{ij}$ are the inner product invariants. Since $\det^2=f_{11}f_{22}-f_{12}^2$, $\I(\quot YG)$ is  the radical of the ideal generated by $f_{12}$ and $f_{11}f_{22}$. The null cone has codimension 2. Again this shows that the codimension condition in Theorem \ref{thm:codimd} is sharp.
\end{example}

We now have the following corollary of Lemma \ref{lem:referee}
 
 \begin{corollary}\label{cor:cofreereduced}
Suppose that $(V,G)$ is cofree and that $Y\subset V$ is $G$-stable such that $Y_{\pr}$ is dense in $Y$. Then $Y$ is $G$-reduced.
\end{corollary}

\begin{remark} For $Y$ to be $G$-reduced, it is not sufficient that every slice representation of $V$ is coreduced.  (This is the same as saying that every fiber of $\pi\colon V\to\quot VG$ is reduced.)\ Just consider Example \ref{ex:noncofree} again.  Here $Y_\pr$ is dense in $Y$ but $Y$ is not $G$-saturated, let alone $G$-reduced.
\end{remark}

\begin{theorem}\label{thm:reduced}
Let $V$ be a $G$-module and let   $Y\subset V$ be $G$-saturated such that $Y_\pr$ is dense in $Y$. Suppose that $\pi(Y)\subset\quot VG$ is the zero set  of $f_1,\dots,f_k$ where the codimension of $Y$ is $k$. Then $Y$ is $G$-reduced.
\end{theorem}

\begin{proof}
The rank of the differential of  $f=(f_1,\dots,f_d)\colon V\to\C^d$  is maximal at a point of each irreducible component of $Y$ since $Y$ is reduced at all points of $Y_\pr$. Thus we can apply Serre's criterion (Proposition \ref{prop:Serre}).  
\end{proof}

\begin{example}\label{ex:Greduced}
Let $(V,G)=(4\C^2,\SL_2(\C))$ and let $Y$ be the zero set of two of the determinant invariants $\det_{ij}$. Then $Y_\pr$ is dense in $Y$ since the only non-principal stratum is $\NN(V)$ which has codimension 3 while $Y$ has codimension 2. By Theorem \ref{thm:reduced}, $Y$ is $G$-reduced.
\end{example}

\section{The case where $Y_\pr$ or $Y_\apr$ is not dense in $Y$}\label{sec:unstable} We can say something in the case that $Y_\apr$ or $Y_\pr$ is not dense in $Y$. We are certainly in this  case if $V$ is not stable, since then $V_\pr$ and $V_\apr$ are empty.  Let  $v\in Y$ such that  $Gv$ is closed. Let   $(N_v,G_v)$ be the slice representation and $S$ the corresponding stratum of $V$. We say that $(N_v,G_v)$ is a \emph{generic slice representation for $Y$\/} if $S\cap Y$ is dense in an irreducible component of $Y$. We also say that \emph{$S$ is  generic  for $Y$}. 

\begin{proposition}\label{prop:slicecoreduced}
Let  $(N_v,G_v)$ be a  generic slice representation of $Y$ corresponding to the stratum $S$ of $V$.   If $Y$ is  $G$-saturated, then $Y\cap S=\pi\inv(\pi(Y\cap S))$. If $Y$ is $G$-reduced, then $(N_v,G_v)$ is coreduced. 
\end{proposition}

\begin{proof}   If $Y$ is $G$-saturated, then we obviously must have that $Y\cap S=\pi\inv(\pi(Y\cap S))$. Let  $Z$ denote $\pi(S)$. Then $\pi\inv (Z)\to Z$ is a fiber bundle with fiber $G\times^{G_v}\NN(N_v)$. If $Y$ is $G$-reduced, then the bundle is reduced, hence  $(N_v,G_v)$ is coreduced.
\end{proof}

Let $S$ be a stratum of $V$. We say that $Y$ is \emph{$S$-saturated\/} if $Y\cap S=\pi\inv(\pi(Y\cap S))$.  We say that $Y$ is \emph{$S$-reduced\/} if $Y$ is $S$-saturated and the slice representation $(N_v,G_v)$ associated to $S$ is coreduced. Corresponding to   Corollaries \ref{cor:cofreealmost} and \ref{cor:cofreereduced}  and Theorems \ref{thm:codimd} and   \ref{thm:reduced} we have the following result whose proof we leave to the reader.

\begin{theorem}
Let $Y\subset V$ be a $G$-stable algebraic set.
\begin{enumerate}
\item If $V$ is cofree and $Y$ is $S$-saturated for every stratum $S$ which is generic for $Y$, then $Y$ is $G$-saturated.
\item If $V$ is cofree and $Y$ is $S$-reduced for every stratum $S$ which is generic for $Y$, then $Y$ is $G$-reduced.
\item  Suppose that $Y$ is $S$-saturated for every every generic stratum $S$ of $Y$. Further assume that the minimal codimension  of  the strata of $V$ which intersect  $Y$ but are not generic for $Y$ is greater than $k$ and that  $\quot YG$ is the zero set of $f_1,\dots,f_k$.  Then $Y$ is $G$-saturated.
\item Suppose that $Y$ is $G$-saturated and that the ideal of $\pi(Y)\subset\quot VG$ is generated by $f_1,\dots,f_k$ where the codimension of $Y$ in $V$ is $k$. Also assume that $Y$ is $S$-reduced for every generic stratum $S$ of $Y$. Then $Y$ is $G$-reduced.
\end{enumerate}
\end{theorem}

\section{The real case}\label{sec:real}

Let $W$ be a real $K$-module where $K$ is compact. Let $X\subset W$ be real algebraic and $K$-stable. Now $K$ is naturally a real algebraic group and the action on $W$ is real algebraic. Moreover, every orbit of $K$ in $W$ is a real algebraic set \cite{SchAlgebraicquotients}. 
Let $Y:=X_\C$ denote the complexification of $X$ inside $V:=W\otimes_\R \C$ and let $G$ denote the complexification $K_\C$ of $K$. Then $G$ is reductive and $V$ is a stable $G$-module (\cite{LunaClosed} or \cite[Cor.\ 5.9]{SchLifting}).  We say that a slice representation $(N_w,K_w)$ is a \emph{generic slice representation for $X$\/} if $w\in X$ and the corresponding stratum contains a nonempty open subset of $X$. Equivalently, the complexification of $(N_w,K_w)$ is generic for $Y$.

\begin{proposition} \label{prop:complexify}\label{prop:aprdense}
\begin{enumerate}
\item $X$ is almost $K$-reduced if and only if $Y$ is $G$-saturated.
\item $X$ is $K$-reduced if and only if $Y$ is $G$-reduced.
\item The set $X_\apr$ (resp.\ $X_\pr$) is   dense in $X$ if and only if the set $Y_\apr$ (resp.\ $Y_\pr$) is   dense in $Y$.
\item $X$ is almost $K$-reduced implies that $X_\apr$ is dense in $X$.
\item $X$ is $K$-reduced implies that $X_\pr$ is dense in $X$.
\end{enumerate}
\end{proposition}

\begin{proof}
The ideal of $Y$ is $\I(X)\otimes_\R\C\subset \R[W]\otimes_\R\C=\C[V]$ and $\I_K(X)\otimes_\R\C=\I_G(Y)$. Thus $\I(Y)=\I_G(Y)$ if and only if $\I(X)=\I_K(X)$, and $\I(Y)=\sqrt{\I_G(Y)}$ if and only if $\I(X)=\sqrt{\I_K(X)}$. Hence we have (1) and (2). 
For (3), note that $X_\apr$ is   open in $X$ and that $Y_\apr$ is   open in $Y$. If a stratum $S$ of $W$ is   dense in an irreducible component of $X$, then the corresponding stratum $S_\C$ of $V$ is   dense in an irreducible component of $Y$. Thus if   $X_\apr$ is not   dense in $X$, then $Y_\apr$ is not   dense in $Y$. Clearly, if $X_\apr$ is dense in $X$, $Y_\apr\supset X_\apr$ is dense in $Y$.  The argument for $X_\pr$ and $Y_\pr$ is similar, hence we have (3).
Now suppose that $X$ is almost $K$-reduced. Then for $S$ a generic stratum of $X$ and $x\in S\cap X$, the  complexification  $Gx\simeq G/G_x$ of $Kx$ is Zariski dense in the fiber $G\times^{G_x}\NN(W_x\otimes_\R\C)$ where $G_x=(K_x)_\C$. Thus $\NN(W_x\otimes_\R\C)$ is a point, i.e., the stratum consists of almost principal orbits. Hence we have (4), and (5) is proved similarly.
\end{proof}

\begin{corollary}\label{cor:orbit}
Let $X=Kw$ be an orbit. Then $X$ is almost $K$-reduced if and only if $Kw$ is almost principal and $X$ is $K$-reduced if and only if $Kw$ is principal.
\end{corollary}

 Unfortunately, it is not true that $X=X_\pr$ (or $X=X_\apr$) implies the same equality for $Y$.

 \begin{example}\label{ex:bad}
 Let $K=\SU_2(\C)$ and $W=2\C^2\oplus\R$ where $K$ acts as usual on the copies of $\C^2$ and trivially on $\R$. We consider $W$ to be $2\H\oplus\R$ where $\H$ denotes the quaternions. Then $K\simeq S^3$, the unit quaternions, and the action on $2\H$ is given by $k(p,q)=(kp,kq)$, $p$, $q\in\H$, $k\in S^3$. Let $p\mapsto \bar p$ denote the usual conjugation of quaternions. The invariants of $K$ acting on $2\H$ are generated by $(p,q)\mapsto (\bar p p,\bar q q,\bar q p)$ where the first two invariants lie in $\R$ and the last in $\H$.  Let $\alpha$ and $\beta$ denote the first two invariants and let $\gamma$ be the real part of $\bar qp$. Let $\delta$, $\epsilon$ and $\zeta$ be the invariants which are the $i$, $j$ and $k$ components of $\bar qp$, respectively. Then there are certainly points in $2\H$ where $\delta$, $\epsilon$ and $\zeta$ vanish and where $\alpha=\beta=\gamma$ is any positive real number. Let $x$ be a coordinate on the copy of $\R$ in $W$ and let $X$ be the subset of $W$ defined by $\delta=\epsilon=\zeta=0$, $\alpha=\beta=\gamma$ and $(\alpha-1)^2+x^2=1/2$. Then $\alpha$ never vanishes on $X$ which implies that the isotropy group at the corresponding point of $W$ is trivial, so we have that $X=X_\pr$. The quotient $X/K$ is a smooth curve, hence $X$ is smooth of dimension $4$. The complexification $Y$ of $X$ also has dimension four and contains some of the points  $(s,t,\pm\sqrt{-3/4})$ where $(s,t)$ lies in the null cone of $2\H\otimes_\R\C\simeq 4\C^2$ for the action of $K_\C\simeq\SL_2(\C)$. But this null cone has dimension 5. Hence $Y$ is not $G$-saturated, let alone $G$-reduced, and $Y\neq Y_\apr$. Moreover,  $X$ is neither $K$-reduced nor almost $K$-reduced.
 \end{example}

Now we   recover the theorems of the introduction. Theorem \ref{thm:mainreal2} is just Corollary \ref{cor:orbit}. Theorem \ref{thm:complexif} is a consequence of Proposition  \ref{prop:complexify}  and Theorem \ref{thm:complex} follows from 
 Theorems \ref{thm:codimd} and \ref{thm:reduced}.

\begin{proof}[Proof of  Theorem \ref{thm:mainreal1}] Suppose that $X$ is $K$-reduced. Then Proposition \ref{prop:complexify} shows that $X_\pr$ is dense in $X$. Conversely, if $(W,K)$ is cofree (equivalently, $(V,G)$ is cofree) and $X_\pr$ is dense in $X$, then $Y_\pr$ is dense in $Y$ by Proposition \ref{prop:aprdense} and $Y$ is $G$-reduced by Corollary \ref{cor:cofreereduced}. Hence $X$ is $K$-reduced. The proof in the almost $K$-reduced case is similar.
\end{proof}


\def\cprime{$'$}
\providecommand{\bysame}{\leavevmode\hbox to3em{\hrulefill}\thinspace}
\providecommand{\MR}{\relax\ifhmode\unskip\space\fi MR }
\providecommand{\MRhref}[2]{%
  \href{http://www.ams.org/mathscinet-getitem?mr=#1}{#2}
}
\providecommand{\href}[2]{#2}

\end{document}